\documentclass{article}
\usepackage{amsmath}
\usepackage{amsfonts}
\usepackage{amsthm}
\usepackage{latexsym}
\usepackage{amssymb}
\usepackage{verbatim}
\usepackage{enumerate}
\font\logic=msam10 at 10pt
\newcommand{\forces}{\mbox{\logic\char'015}}
\newcommand{\restrict}{\mbox{\logic\char'026}}

\newcommand{\less}{\mathord{<}}

\newcommand{\cf}{\operatorname{cf}}
\newcommand{\dom}{\operatorname{dom}}

\newtheorem{thrm}{Theorem}[section]
\newtheorem{lem}[thrm]{Lemma}
\newtheorem{cor}[thrm]{Corollary}

\newtheoremstyle{hdefinition}%
  {\topsep}%
  {\topsep}%
  {\upshape}
  {}%
  {\bfseries}%
  {.}
  { }%
  {\thmnumber{#2 }\thmname{#1}\thmnote{ \rm(#3)}}%

\newtheoremstyle{hclaim}%
  {\topsep}%
  {\topsep}%
  {\itshape}%
  {}%
  {\bfseries}%
  {.}
  { }%
  {\thmname{#1}\thmnote{ \rm#3}}%

\newtheoremstyle{hnotation}%
  {\topsep}%
  {\topsep}%
  {\upshape}%
  {}%
  {\bfseries}%
  {.}
  { }%
  {\thmname{#1}\thmnote{ \rm#3}}%

\theoremstyle{hclaim}
\newtheorem*{claim*}{Claim}

\theoremstyle{hdefinition}

\newtheorem{ques}[thrm]{Question}

\theoremstyle{hclaim}
\newtheorem{claim}{Claim}

\theoremstyle{hnotation}
\newtheorem{notation}{Notation}

\begin{document}

\title{Splitting stationary sets from weak forms of Choice
\thanks{
The first author is supported in part by NSF grant DMS-0401603. The
research of the second author is supported by the United
States-Israel Binational Science Foundation.
The research in this paper began during a visit by the
first author to Rutgers University in October 2007, supported by NSF
grant DMS-0600940. This is the second author's publication number
925.}}

\author{Paul Larson and Saharon Shelah}


\pagenumbering{arabic}
\maketitle

\begin{abstract} Working in the context of restricted forms of the
Axiom of Choice, we consider the problem of splitting the ordinals
below $\lambda$ of cofinality $\theta$ into $\lambda$ many
stationary sets, where $\theta < \lambda$ are regular cardinals.
This is a continuation of \cite{Sh835}.
\end{abstract}

In this note we consider the issue of splitting stationary sets in
the presence of weak forms of the Axiom of Choice plus the existence
of certain types of ladder systems. Our primary interest is the
theory ZF + DC plus the assertion that for some large enough
cardinal $\lambda$, there is a ladder system for the members of
$\lambda$ of countable cofinality, that is, a function that assigns
to every such $\alpha < \lambda$ a cofinal subset of ordertype
$\omega$. In this context, we show that for every $\gamma < \lambda$
of uncountable cofinality the set of $\alpha< \gamma$ of countable
cofinality can be uniformly split into $\cf(\gamma)$ many stationary
sets. It follows from this and the results of \cite{Sh835} that
there is no nontrivial elementary embedding from $V$ into $V$, under
the assumption of ZF + DC plus the assertion that the countable
subsets of each ordinal can be wellordered. As a counterpoint to
some of the results presented here, we give a symmetric forcing
extension in which there are regressive functions on stationary sets
not constant on stationary sets.

\section{AC and DC}

Given a nonempty set $Z$, the statement AC$_{Z}$ says that whenever
$\langle X_{a} : a \in Z \rangle$ is a collection of nonempty sets,
there is a function $f$ with domain $Z$ such that $f(a) \in X_{a}$
for each $a \in Z$. If $\gamma$ is an ordinal, the statement
AC$_{\less\gamma}$ says that AC$_{\eta}$ holds for all ordinals
$\eta < \gamma$.

A \emph{tree} $T$ is a set of functions such that the domain of each
function is an ordinal, and such that, whenever $f \in T$ and
$\alpha \in \dom(f)$, $f \restrict \alpha \in T$. Two elements $f$,
$g$ of a tree $T$ are \emph{compatible} if $f \subseteq g$ or $g
\subseteq f$. A \emph{branch} through a tree is a pairwise compatible
collection of elements of $T$. A branch is \emph{maximal} if it is
not properly contained in any other branch.

Given an ordinal $\gamma$, the statement DC$_{\gamma}$ says that for
every tree $T \subseteq {^{\less\gamma}X}$ (for some set $X$) there is
a $b \subseteq T$ which is a maximal branch. The statement
DC$_{\less\gamma}$ says that DC$_{\eta}$ holds for all ordinals
$\eta < \gamma$. It follows immediately from the definition of
DC$_{\gamma}$ that DC$_{\gamma}$ implies DC$_{\eta}$ for all $\eta <
\gamma$. We write DC for DC$_{\omega}$ and AC for the statement that
AC$_{Z}$ holds for all sets $Z$.

Lemma \ref{dcac} shows that DC$_{\gamma}$ implies AC$_{\gamma}$ for
all ordinals $\gamma$.

\begin{lem}\label{dcac} Suppose that $\gamma$ is a limit ordinal such that
DC$_{\gamma}$ holds, and $T$ is a tree such that
\begin{itemize}
\item every $f \in T$ is a function with domain $\eta$, for some
$\eta < \gamma$;
\item for all limit ordinals $\eta < \gamma$, if $f$ is a function
with domain $\eta$ such that $f \restrict \alpha \in T$ for all
$\alpha \in \eta$, then $f \in T$;
\item for every $f \in T$ there is a $g \in T$ properly containing
$f$.
\end{itemize}
Then there is a function $f$ with domain $\gamma$ such that $f
\restrict \alpha \in T$ for all $\alpha < \gamma$.
\end{lem}

\begin{proof} Let $b$ be a maximal branch of $T$, and let $f =
\bigcup b$. Then $f$ is a function whose domain is an ordinal $\eta
\leq \gamma$. If $\eta < \gamma$, then $f \in T$ and $f$ has a
proper extension in $T$, contradicting its supposed maximality.
\end{proof}

\section{Ladder systems}

\begin{notation} Given an ordinal $\delta$, we let $\cf(\delta)$ denote the
cofinality of $\delta$. Given an ordinal $\alpha$ and a set $A$, we
let $C^{\alpha}_{A}$ denote the ordinals below $\alpha$ whose
cofinality is in $A$. Given an ordinal $\lambda$ and a function $f$,
we let $\phi(\lambda, f)$ be the statement that there exists a
sequence $\langle c_{\delta} : \delta \in C^{\lambda}_{\dom(f)}
\rangle$ such that each $c_{\delta}$ is a cofinal subset of $\delta$
of ordertype less than $f(\cf(\delta))$.
\end{notation}

Note that $\phi(\lambda, f)$ implies that $f(\gamma) \geq \gamma +
1$ for all regular cardinals $\gamma \in \dom(f)$.

\begin{notation} We let $\psi(\lambda, \theta)$ be the statement $\phi(\lambda,
\{(\theta, \theta+1)\})$. We say that a sequence $\langle c_{\delta}
: \delta \in C^{\lambda}_{\dom(f)} \rangle$ \emph{witnesses}
$\phi(\lambda, f)$ if each $c_{\delta}$ is a cofinal subset of
$\delta$ of ordertype less than $f(\cf(\delta))$, and similarly for
$\psi(\lambda, \theta)$.
\end{notation}

The statement $\psi(\lambda, \omega)$ follows from the statement
Ax$^{2}_{\lambda}$ of \cite{Sh835} (in the case $\partial =
\omega$), which says that there exists a well-orderable $\mathcal{A}
\subseteq [\lambda]^{\aleph_{0}}$ such that every element of
$[\lambda]^{\aleph_{0}}$ has infinite intersection with a member of
$\mathcal{A}$. We will be primarily interested in statements
$\phi(\lambda, f)$ where $f$ is either the ordinal successor
function or the cardinal successor function on some set of regular
cardinals. The two following lemmas show that when the domain of $f$
is a single regular cardinal, there is in some sense no statement
strictly in between these two.

\begin{lem}[ZF]\label{omegathin} For each ordinal ordinal $\gamma$ there exists a
sequence $$\langle e_{\delta} : \delta < \gamma\rangle$$ such that
each $e_{\delta}$ is a cofinal subset of $\delta$ of ordertype less
than or equal to $|\gamma|$.
\end{lem}

\begin{proof} Let $\pi \colon |\gamma| \to \gamma$ be a bijection. For each
$\delta < \gamma$, let $e_{\delta}$ be the set of ordinals of the
form $\pi(\alpha)$, where $\alpha < |\gamma|$, $\pi(\alpha) <
\delta$ and $\pi(\alpha)
> \pi(\beta)$ for all $\beta < \alpha$ with $\pi(\beta) < \delta$.
\end{proof}

\begin{notation} Given a set $x$ of ordinals, we let $o.t.(x)$ denote the ordertype of $x$.
Given an ordinal $\eta < o.t.(x)$, we let $x(\eta)$ be the $\eta$-th
member of $x$, i.e., the unique $\alpha \in x$ such that $o.t.(x
\cap \alpha) = \eta$.
\end{notation}

\begin{lem}[ZF] Let $\lambda$ be an ordinal, let $\theta$ be a regular cardinal,
and let $\eta$ be an ordinal less than $\theta^{+}$. Then
$\phi(\lambda, \{(\theta, \eta)\})$ implies $\psi(\lambda, \theta)$.
\end{lem}

\begin{proof} Let $\langle c_{\delta} : \delta \in
C^{\lambda}_{\{\theta\}} \rangle$ witness $\phi(\lambda, \{(\theta,
\eta)\})$, and let $\langle e_{\delta} : \delta < \eta\rangle$ be
such that each $e_{\delta}$ is a cofinal subset of $\delta$ of
ordertype less than or equal to $\theta$. For each $\delta \in
C^{\lambda}_{\{\theta\}}$, letting $\alpha_{\delta}$ be the
ordertype of $c_{\delta}$, let $d_{\delta} = \{c_{\delta}(\beta)
\mid \beta \in e_{\alpha_{\delta}}\}$. Then each $d_{\delta}$ is a
cofinal subset of $\delta$ of ordertype $\theta$.
\end{proof}





\section{Splitting $C^{\lambda}_{\theta}$ from
DC$_{\theta}$ and AC$_{\gamma}$}

\begin{notation} Given ordinals $\alpha, \beta, \eta$ and a sequence of sets of
ordinals $\bar{C} = \langle c_{\delta} : \delta \in S\rangle$ (for
some set $S$), we let $S^{\eta}_{\alpha, \beta}(\bar{C})$ be the set
of $\delta \in S$ such that $o.t.(c_{\delta}) > \eta$ and
$c_{\delta}(\eta) \in [\alpha, \beta)$.
\end{notation}




We are primarily interested in the following theorem in the case
where $\theta$ and $\gamma$ are both $\omega$, in which case
$\psi(\lambda, \omega)$ implies the existence of a sequence
$\bar{C}$ satisfying the stated hypotheses.

\begin{thrm}[ZF]\label{claimonelem} Suppose that the following hold.
\begin{itemize}
\item $\theta \geq \aleph_{0}$ is a regular cardinal such that DC$_{\theta}$ holds;
\item $\gamma\geq \theta$ is an ordinal such that AC$_{\gamma}$ holds;
\item $\lambda$ is an ordinal of cofinality greater than $\gamma$;
\item $E$ is a club subset of $\lambda$; \item $\bar{C} = \langle
c_{\delta} : \delta \in C^{\lambda}_{\{\theta\}} \cap E \rangle$ is
a sequence such that each $c_{\delta}$ is a cofinal subset of
$\delta$ of ordertype less than or equal to $\gamma$. \end{itemize}
Then \begin{enumerate} \item there exists an $\eta^{*} <\gamma$ such
that for each $\alpha < \lambda$ there exists a $\beta \in (\alpha,
\lambda)$ such that $S^{\eta^{*}}_{\alpha, \beta}(\bar{C})$ is a
stationary subset of $\lambda$;
\item there exist
functions $g \colon C^{\lambda}_{(\gamma,\lambda)} \to \gamma$, $h
\colon C^{\lambda}_{(\gamma,\lambda)} \to \lambda$ and a collection
of ordinals $\langle \alpha^{\xi}_{\beta} : \xi \in
C^{\lambda}_{(\gamma,\lambda)}, \beta < h(\xi)\rangle$ such that
\begin{itemize}
\item for each $\xi \in C^{\lambda}_{(\gamma,\lambda)}$, $h(\xi) < \cf(\xi)^{+}$;
\item for each $\xi \in C^{\lambda}_{(\gamma,\lambda)}$, $\langle \alpha^{\xi}_{\beta} : \beta <
h(\xi) \rangle$ is a continuous increasing sequence cofinal in
$\xi$,
\item for each $\xi \in C^{\lambda}_{(\gamma,\lambda)}$ and each $\beta < h(\xi)$,
$S^{g(\xi)}_{\alpha^{\xi}_{\beta}, \alpha^{\xi}_{\beta + 1}}(\bar{C}
\restrict \xi)$ is stationary.
\end{itemize}
\end{enumerate}
\end{thrm}

\begin{proof} We prove the first part first.
Supposing that there is no such $\eta^{*}$, for each $\eta < \gamma$
let $\alpha^{*}_{\eta} < \lambda$ be the least $\alpha < \lambda$
such that $S^{\eta}_{\alpha, \beta}(\bar{C})$ is nonstationary for
all $\beta \in (\alpha, \lambda)$. Using the fact that $\cf(\lambda)
> \gamma$, let $\alpha^{*}$ be the least element of $C^{\lambda}_{\{\theta\}}
\cap E$ greater than or equal to the supremum of
$\{\alpha^{*}_{\eta} : \eta < \gamma \}$. Now, applying
DC$_{\theta}$ and AC$_{\gamma}$, we choose a continuous increasing
sequence of ordinals $\langle \alpha_{\xi} : \xi < \theta \rangle$
and sets $D_{\xi,\eta}$ $(\xi < \theta, \eta < \gamma)$ by recursion
on $\xi < \theta$ such that
\begin{enumerate}
\item $\alpha_{0} = \alpha^{*}$
\item each $D_{\xi,\eta}$ is a club subset of $E$ disjoint from
$S^{\eta}_{\alpha^{*}, \alpha_{\xi}}(\bar{C})$
\item if $\xi < \theta$ a limit ordinal then
$\alpha_{\xi} = \bigcup\{\alpha_{\zeta} : \zeta < \xi\}$
\item if $\xi = \zeta + 1$, then $\alpha_{\xi} = \min(\bigcap_{\rho
\leq \zeta, \eta < \gamma} D_{\rho, \eta} \setminus (\alpha_{\zeta}
+ 1))$
\end{enumerate}

Let $\alpha_{\theta} = \bigcup\{ \alpha_{\xi} : \xi < \theta \}$.
Then $\alpha_{\theta} < \lambda$ as $\cf(\lambda) > \theta$, so
$\alpha_{\theta} \in C^{\lambda}_{\{\theta\}} \cap E$.
For some $\eta < \gamma$, $c_{\alpha_{\theta}}(\eta)
> \alpha^{*}$, hence for some $\xi < \theta$, $c_{\alpha_{\theta}}(\eta) \in
[\alpha_{*}, \alpha_{\xi})$. Then $\alpha_{\theta} \in
S^{\eta}_{\alpha_{0}, \alpha_{\xi}}(\bar{C})$, contradicting the
assumption that $\alpha_{\theta} \in D_{\xi,\eta}$.

To prove the second part of the lemma, fix $\xi \in
C^{\lambda}_{(\gamma,\lambda)}$. Applying the first part of the
lemma with $\xi$ as $\lambda$, let $g(\xi)$ be the least $\eta \in
\gamma$ such that for each $\alpha < \xi$ there exists a $\beta \in
(\alpha, \xi)$ such that $S^{\eta}_{\alpha, \beta}(\bar{C}\restrict
\xi)$ is a stationary subset of $\xi$. Then by recursion on $\beta <
\xi$ we can choose an increasing continuous sequence of ordinals
$\alpha^{\xi}_{\beta} < \lambda$ $(\beta < \xi)$ such that
$\alpha^{\xi}_{0} = 0$, $\alpha^{\xi}_{\beta} =
\cup\{\alpha^{\xi}_{\zeta}: \zeta < \beta\}$ for limit $\beta$, and,
if $\beta = \zeta + 1$, then, if $\alpha^{\xi}_{\zeta} = \xi$ then
$\alpha^{\xi}_{\beta} = \xi$, otherwise $\alpha^{\xi}_{\beta}$ is
the minimal ordinal $\delta \in (\alpha^{\xi}_{\zeta}, \xi)$ such
that $S^{g(\xi)}_{\alpha^{\xi}_{\zeta}, \delta}(\bar{C}\restrict
\xi)$ is stationary. Let $h(\xi)$ be the least $\beta$ such that
$\alpha^{\xi}_{\beta} = \xi$ if some such $\beta$ exists, and let it
be $\xi$ otherwise. Since there is a club subset of $\xi$ of
cardinality $\cf(\xi)$, and the sets
$S^{g(\xi)}_{\alpha^{\xi}_{\beta}, \alpha^{\xi}_{\beta +
1}}(\bar{C}\restrict \xi)$ ($\beta < h(\xi)$) are disjoint
stationary subsets of $\xi$, $h(\xi) < \cf(\xi)^{+}$. This completes
the definitions of $g$, $h$ and $\langle \alpha^{\xi}_{\beta} : \xi
\in C^{\lambda}_{(\gamma,\lambda)}, \beta < h(\xi)\rangle$.
\end{proof}

\begin{cor} Suppose that the following hold.
\begin{itemize}
\item $\theta \geq \aleph_{0}$ is a regular cardinal such that DC$_{\theta}$ holds;
\item
$\lambda$ is an ordinal of cofinality greater than $\theta$; \item
$A$ is the set of regular cardinals in the interval $[\theta,
\lambda)$.
\end{itemize}
Then $\psi(\lambda, \theta)$ implies $\phi(\lambda, f)$, where $f$
is the cardinal successor function on $A$.
\end{cor}

The following is a consequence of the results of \cite{Sh835},
Woodin's proof of Kunen's Theorem (see \cite{Kan03}) and the
arguments in this section.

\begin{cor}[ZF + DC] Assume that for every ordinal $\lambda$ there exists a
wellorderable set $\mathcal{A} \subseteq [\lambda]^{\aleph_{0}}$ such
that every element of $[\lambda]^{\aleph_{0}}$ has infinite
intersection with a member of $\mathcal{A}$. Then there is no
nontrivial elementary embedding from $V$ into $V$.
\end{cor}

\begin{proof}
Suppose towards a contradiction that $j \colon V \to V$ is an
elementary embedding. Let $\kappa_{0}$ be the critical point of $j$,
and for each nonzero $n < \omega$, let $\kappa_{n+1} =
j(\kappa_{n})$.
Let $\kappa_{\omega} = \cup \{ \kappa_{n} : n < \omega \}$. Then
$j(\kappa_{\omega}) = \kappa_{\omega}$ and $j(\kappa_{\omega}^{+}) =
\kappa_{\omega}^{+}$. For no $\alpha < \kappa_{0}$ is there is a
surjection from $V_{\alpha}$ onto $\kappa_{0}$ (to see this,
consider $j(\pi)$, where $\pi$ is such a surjection, in light of the
fact that $j \restrict V_{\kappa_{0}}$ is the identity function). By
elementarity, then, the same is true for each $\kappa_{n}$, and so
the same is true for $\kappa_{\omega}$. Then by the results of
\cite{Sh835} (specifically, Lemma 2.13),  $\kappa_{\omega}^{+}$ is
regular.


Let $\bar{C} = \langle c_{\delta} : \delta \in
C^{\kappa_{\omega}^{+}}_{\{\omega\}}\rangle$ witness
$\psi(\kappa_{\omega}^{+}, \omega)$. Applying Theorem
\ref{claimonelem}, let $n_{*} \in \omega$ and $\bar{\alpha} =
\langle \alpha_{\xi} : \xi < \kappa_{\omega}^{+} \rangle$ be such
that $\bar{\alpha}$ is a continuous increasing sequence of elements
of $\kappa_{\omega}^{+}$ and such that $S^{n_{*}}_{\alpha_{\xi},
\alpha_{\xi + 1}}(\bar{C})$ is a stationary subset of
$C^{\kappa_{\omega}^{+}}_{\{\omega\}}$ for each $\xi < \kappa_{\omega}^{+}$.


Let $F$ be the set of limit ordinals $\delta < \kappa_{\omega}^{+}$
such that $j(\alpha) < \delta$ for every $\alpha < \delta$. Then $F$
is a club. Let $E$ be the set of members of $F$ of cofinality less
than $\kappa_{0}$. Then $j \restrict E$ is the identity function,
and no stationary subset of $C^{\kappa_{\omega}^{+}}_{\{\omega\}}$ is
disjoint from $E$.




Let $\langle S'_{\xi} : \xi < \kappa_{\omega}^{+}\rangle = j(\langle
S^{n_{*}}_{\alpha_{\xi}, \alpha_{\xi + 1}}(\bar{C}) : \xi <
\kappa_{\omega}^{+}\rangle)$.
As $j$ is an elementary embedding, $V \models ``S'_{\kappa_{0}}$ is
a stationary subset of $C^{\kappa_{\omega}^{+}}_{\{\omega\}}$ disjoint from $S'_{\xi}$ for $\xi
\in \kappa_{\omega}^{+} \setminus \{\kappa_{0}\}$." Hence,
$S'_{\kappa_{0}}$ is disjoint from $S'_{j(\xi)}$, for all $\xi <
\kappa_{\omega}^{+}$. But
$$\bigcup_{\xi < \kappa_{\omega}^{+}}S'_{j(\xi)} \supset \bigcup_{\xi <
\kappa_{\omega}^{+}}(S'_{j(\xi)} \cap E) = \bigcup_{\xi <
\kappa_{\omega}^{+}} (S^{n_{*}}_{\alpha_{\xi}, \alpha_{\xi + 1}}(\bar{C})
\cap E) = E \cap C^{\kappa_{\omega}^{+}}_{\{\omega\}}.$$

\end{proof}

\section{Club guessing}


In this section we show that the standard club-guessing arguments go
through under weak forms of Choice plus the existence of ladder
systems. Theorem \ref{guessingdc} uses forms of DC, and Theorem
\ref{guessingac} uses AC.

\begin{thrm}[ZF]\label{guessingdc}
Let $\theta < \lambda$ be regular cardinals, with $\theta^{+} <
\lambda$, and suppose that DC$_{\theta^{+}}$ holds. Suppose that
$\langle c_{\delta} : \delta \in C^{\lambda}_{\{\theta\}}\rangle$ is
a sequence such that each $c_{\delta}$ is a closed cofinal subset of
$\delta$ of ordertype less than $\theta^{+}$. Then the following
hold.

\begin{itemize}
\item There exists a
sequence $\langle d_{\delta} : \delta \in
C^{\lambda}_{\{\theta\}}\rangle$ such that each $d_{\delta}$ is
a cofinal subset of $\delta$, and such that for every club subset $D
\subseteq \lambda$ there is a $\delta \in
C^{\lambda}_{\{\theta\}}$ with $d_{\delta} \subseteq D$.

\item If $\theta$ is uncountable, then there exists a
sequence $\langle d_{\delta} : \delta \in
C^{\lambda}_{\{\theta\}}\rangle$ such that each $d_{\delta}$ is a
closed cofinal subset of $c_{\delta}$, and such that for every club
subset $D \subseteq \lambda$ there is a $\delta \in
C^{\lambda}_{\{\theta\}}$ with $d_{\delta} \subseteq D$.

\end{itemize}
\end{thrm}

\begin{proof} We argue as in \cite{Shg}, Chapter III.

For the first part, for any two sets $A$, $B$, let $gl(A, B)$ denote
the set $$\{ \sup(\alpha \cap B) \mid \alpha \in A \setminus
(\min(B) + 1)\}.$$ Note that if $A$ and $B \cap \gamma$ are club
subsets of an ordinal $\gamma$, then $gl(A, B)$ is a club subset of
$B \cap \gamma$ as well.

Supposing that the first conclusion of the theorem is false, choose
for each $\zeta \leq \theta^{+}$ a club subset $D_{\zeta} \subseteq
\lambda$ such that the following conditions are satisfied.
\begin{itemize}
\item $D_{0}$ does not contain $c_{\delta}$ for any $\delta \in
C^{\lambda}_{\{\theta\}}$; \item for each $\zeta < \theta^{+}$,
$D_{\zeta + 1}$ is contained in the limit points of $D_{\zeta}$, and
$D_{\zeta + 1}$ does not contain $gl(c_{\delta}, D_{\zeta})$ for any
$\delta \in C^{\lambda}_{\{\theta\}}$ which is a limit point of
$D_{\zeta}$.
\item for each limit ordinal $\zeta \leq \theta^{+}$, $D_{\zeta} =
\bigcap_{\xi < \zeta}D_{\xi}$.
\end{itemize}
Now fix a $\delta \in C^{\lambda}_{\{\theta\}}$ which is a limit
point of $D_{\theta^{+}}$. For each $\alpha \in c_{\delta}$, either
there is a $\zeta < \theta^{+}$ such that $\alpha \leq
\min(D_{\zeta})$, or $\langle \sup(\alpha \cap D_{\zeta}) : \zeta <
\theta^{+}\rangle$ is a nonincreasing sequence which reaches an
eventually constant value. Since $|c_{\delta}| < \theta^{+}$, there
is a $\zeta < \theta^{+}$ such that for each $\alpha \in
c_{\delta}$, $\alpha > \min(D_{\zeta})$ implies $\alpha >
\min(D_{\zeta + 1})$, and, if $\alpha > \min(D_{\zeta})$, then
$\sup(\alpha \cap D_{\zeta}) = \sup(\alpha \cap D_{\zeta + 1})$.
Then $gl(c_{\delta}, D_{\zeta}) = gl(c_{\delta}, D_{\zeta + 1})$.
However, $gl(c_{\delta}, D_{\zeta + 1}) \subseteq D_{\zeta + 1}$ and
$D_{\zeta +1}$ was chosen not to contain $gl(c_{\delta},
D_{\zeta})$, giving a contradiction.

For the second part, note that we can just take the intersection of
$c_{\delta}$ and $d_{\delta}$ for each $\delta \in
C^{\lambda}_{\{\theta\}}$, where $d_{\delta}$ is given by the first
part.
\end{proof}

\begin{ques} Does DC$_{\theta}$ suffice for Theorem \ref{guessingdc}?
\end{ques}

\begin{thrm}[ZF]\label{guessingac} Suppose that
\begin{itemize}
\item  $\theta < \lambda$ are regular uncountable cardinals;
\item there is no surjection from $\mathcal{P}(\theta)$ onto
$\lambda$; \item AC$_{X}$ holds, where $X$ is the union of
$\theta^{+}$ and the set of club subsets of $\theta$;
\item $\langle c_{\delta} : \delta \in
C^{\lambda}_{\{\theta\}}\rangle$ is a sequence such that each
$c_{\delta}$ is a closed cofinal subset of $\delta$ of ordertype
less than $\theta^{+}$.
\end{itemize}
Then there exists a sequence $\langle e_{\delta} : \delta \in
C^{\lambda}_{\{\theta\}}\rangle$ such that each $e_{\delta}$ is a
closed cofinal subset of $c_{\delta}$ of ordertype $\theta$, and
such that for every club subset $D \subseteq \lambda$ there is a
$\delta \in C^{\lambda}_{\{\theta\}}$ with $e_{\delta} \subseteq D$.
\end{thrm}

\begin{proof} Applying AC$_{X}$, let $\bar{D} = \langle
d_{\delta} : \delta \in C^{\theta^{+}}_{\{\theta\}} \rangle$ be such
that each $d_{\delta}$ is a club subset of $\delta$ of ordertype
$\theta$. For each $\delta \in C^{\lambda}_{\{\theta\}}$, let
$c'_{\delta} = \{ c_{\delta}(\eta) : \eta \in d_{o.t.(c_{\delta})}\}$. Then each $c'_{\delta}$
is a closed, cofinal subset of $\delta$ of ordertype $\theta$.

For each $\delta\in C^{\lambda}_{\{\theta\}}$ and for each club $C
\subseteq \theta$, let $c(C)_{\delta} = \{ c'_{\delta}(\beta)  \mid
\beta \in C\}$.
Supposing that the conclusion fails, choose $\langle E_{C} : C
\subseteq \theta \text{ club} \rangle$ such that each $E_{C}$ is a
club subset of $\lambda$ not containing $c(C)_{\delta}$ for any
$\delta \in C^{\lambda}_{\{\theta\}}$. As there is no surjection
from $\mathcal{P}(\theta)$ to $\lambda$,
$$E = \bigcap\{ E_{C} : C \subseteq \theta \text{ club}\}$$ is a club subset
of $\lambda$. Let $\delta$ be any limit member of $E$ in
$C^{\lambda}_{\{\theta\}}$, and let $$C = \{ \alpha < \theta\mid
c'_{\delta}(\alpha) \in E\}.$$ Then $c(C)_{\delta} = c'_{\delta}
\cap E \subseteq E_{C}$, contradicting the choice of $E_{C}$.
\end{proof}



\section{Splitting at higher cofinalities}

In this section we consider the problem of using a ladder system to
split $C^{\lambda}_{\{\theta\}}$ into stationary sets without the
help of AC and DC. So the difference is that we try to split at
cofinality $\theta$ without DC$_{\theta}$.

\begin{thrm}[ZF]\label{claimtwo} Suppose that the following hold.
\begin{itemize} \item $\theta < \lambda$ are regular uncountable cardinals;
\item  $\gamma \in [\theta, \lambda)$ is an ordinal;
\item $\bar{C} = \langle c_{\delta} : \delta \in
C^{\lambda}_{\{\theta\}}\rangle$ is a sequence such that each
$c_{\delta}$ is a cofinal subset of $\delta$ of ordertype less than
or equal to $\gamma$.
\end{itemize}


Then either \begin{enumerate} \item there exist $\eta < \gamma$ and
a continuous increasing sequence $\langle \alpha_{\xi} : \xi <
\lambda \rangle$ such that each $\alpha_{\xi} \in \lambda$ and each
$S^{\eta}_{\alpha_{\xi}, \alpha_{\xi + 1}}(\bar{C})$ is stationary,
or
\item the following two statements hold:
\begin{enumerate} \item for some club $E \subseteq \lambda$ there exists a
regressive function $F$ on $E \cap C^{\lambda}_{\{\theta\}}$ such
that $F^{-1}\{\beta\}$ not stationary for any $\beta < \lambda$, and
\item if AC$_{\gamma}$ holds, then for
some $\alpha_{*} < \lambda$ there is a regressive function $G$ on
$C^{\lambda}_{\{\theta\}} \setminus (\alpha_{*} + 1)$ such that for
each $\beta < \lambda$ the set of $\gamma \in
C^{\lambda}_{\{\theta\}}$ such that $G(\gamma) < \beta$ is not
stationary.
\end{enumerate}
\end{enumerate}
\end{thrm}

\begin{proof}

Suppose first that there exists a $\eta < \gamma$ such that for each
$\alpha < \lambda$ there exists a $\beta \in (\alpha, \lambda)$ such
that $S^{\eta}_{\alpha, \beta}(\bar{C})$ is a stationary subset of
$\lambda$. Then we can recursively choose $\alpha_{\xi} < \lambda$ $(\xi <
\lambda)$, increasing continuously with $\xi$,
such that $\alpha_{0} = 0$ and
$$\alpha_{\xi + 1} = min\{ \alpha : \alpha_{\xi} < \alpha <
\lambda \wedge S^{\eta}_{\alpha_{\xi}, \alpha}(\bar{C})\text{ is
stationary}\}.$$ Then the first conclusion of the lemma holds.

Suppose instead that there is no such $\eta$. Then for each $\eta <
\gamma$, let $\alpha^{*}_{\eta} < \lambda$ be minimal such that for
all $\beta \in (\alpha^{*}_{\eta}, \lambda)$,
$S^{\eta}_{\alpha^{*}_{\eta}, \beta}(\bar{C})$ not a stationary
subset of $\lambda$. Let $\alpha_{*} = sup\{ \alpha^{*}_{\eta} :
\eta < \gamma \}$. Then $\alpha_{*} < \lambda$, as $\lambda =
\cf(\lambda) > \gamma$.

Define $F \colon C^{\lambda}_{\{\theta\}}\setminus (\alpha_{*} + 1)
\to \lambda \times \gamma$ by letting $F(\delta) = (\alpha, \eta)$
if $\alpha$ is the least element of $c_{\delta}$ greater than
$\alpha_{*}$ and $\alpha = c_{\delta}(\eta)$. Then for no $(\alpha,
\eta) \in \lambda \times \gamma$ is $F^{-1}\{ (\alpha, \eta)\}$
stationary.

Let $H \colon \lambda \times \gamma \to \lambda$ be the function
$H(\alpha, \eta) = \gamma \cdot \alpha + \eta$, and let $E$ be the
set of $\alpha \in (\alpha_{*}, \lambda)$ such that $H(\beta, \eta)
< \alpha$ for all $\beta < \alpha$ and $\eta < \gamma$. Then $E$ is
a club set. Furthermore, the function $H \circ F$ is regressive on
$E \cap C^{\lambda}_{\{\theta\}}$ and not constant on a stationary
set, as desired.

Finally, suppose that AC$_{\gamma}$ holds. For each $\beta \in
(\alpha_{*}, \lambda)$ and each $\eta < \gamma$,
$S^{\eta}_{\alpha_{*}, \beta}(\bar{C})$ is nonstationary. It follows
(from AC$_{\gamma}$) that for each $\beta \in (\alpha_{0},
\lambda)$, $S_{\beta} = \bigcup_{\eta < \gamma}S^{\eta}_{\alpha_{*},
\beta}(\bar{C})$ is nonstationary. Now define $G \colon
C^{\lambda}_{\{\theta\}}\setminus (\alpha_{*} + 1) \to \lambda$ by
letting $G(\delta)$ be the least element of $c_{\delta}$ greater
than $\alpha_{*}$. Then for every $\beta \in \lambda$, the set of
$\delta \in C^{\lambda}_{\{\theta\}} \setminus (\alpha_{*} + 1)$ with
$G(\delta) < \beta$ is nonstationary.
\end{proof}

\section{A model of ZF and a regressive function}

In this section we give a proof of the following theorem, which is
complementary to Theorem \ref{claimtwo}.

\begin{thrm}[ZFC] Let $\theta < \lambda$ be regular cardinals.
There is a partial order $P$ such that in the $P$-extension of $V$
there is an inner model $M$ with the following properties.
\begin{itemize}
\item $M$ and $V$ have the same ordinals of cofinality $\theta$;
\item $\lambda$ is a regular cardinal in $M$;
\item $M$ satisfies ZF + DC$_{\less\theta}$ + $\phi(\lambda,f)$,
where $f$ is the ordinal successor function on the regular cardinals
below $\theta$;
\item there exists in $M$ a regressive function
on $(C^{\lambda}_{[\theta, \lambda)})^{M}$ which is not constant on
a stationary set.
\end{itemize}
\end{thrm}

The strategy for the proof is a direct modification of Cohen's
original proof of the independence of AC (see \cite{Jec03}).

Assume that ZFC holds and that $\theta < \lambda$ are
regular cardinals. Given a set $X \subseteq \lambda \times \lambda$,
let $P_{X}$ be the partial order whose conditions consist of pairs
$(f,d)$ such that
\begin{itemize}
\item $f$ is a partial regressive function on $C^{\lambda}_{[\theta, \lambda)}$
whose domain is $\alpha \cap C^{\lambda}_{[\theta, \lambda)}$ for
some successor ordinal $\alpha< \lambda$;

\item $d$ is a partial function whose domain is a subset of
$X$ of cardinality less than $\lambda$ such that for each
$(\alpha,\beta)$ in the domain of $d$, $d(\alpha, \beta)$ is a
closed, bounded subset of $max(\dom(f))+1$ disjoint from
$f^{-1}\{\alpha\}$.
\end{itemize}
The order on $P_{X}$ is given by: $(f,d) \leq (g, e)$ if $g \subseteq
f$, $\dom(e) \subseteq \dom(d)$ and $d(\alpha,\beta) \cap
(max(\dom(g))+1) = e(\alpha, \beta)$ for all $(\alpha,\beta) \in
\dom(e)$.

The partial order $P_{X}$ is closed under decreasing sequences of
length less than $\theta$ and therefore does not add sets of
ordinals of cardinality less than $\theta$. Furthermore, if
$|X|^{+}<\lambda$, then below densely many conditions (conditions
$(f,d)$ with $|\dom(f)| > |X|$) every descending sequence in $P_{X}$
of length less than $\lambda$ has a lower bound, so $P_{X}$ does not
add sequences from $V$ of length less than $\lambda$. We will see
below that $P_{X}$ is in some sense homogeneous.

Given $X \subseteq \lambda \times \lambda$ and a regressive function
$F$ on $C^{\lambda}_{[\theta,\lambda)}$, let $Q_{F,X}$ denote the
partial order whose conditions are partial functions $d$ with domain
a subset of $X$ of cardinality less than $\lambda$, such that for
each $(\alpha,\beta)$ in the domain of $d$, $d(\alpha,\beta)$ is a
closed, bounded subset of $\lambda$ disjoint from
$F^{-1}\{\alpha\}$. If $X$ is a subset of $\lambda \times \lambda$
such that $|X|^{+} < \lambda$, and $Y\subseteq \lambda \times \lambda$
is disjoint from $X$, then, since $P_{X}$ does not add bounded
subsets of $\lambda$, $P_{X \cup Y}$ is forcing-isomorphic to $P_{X}
* Q_{\dot{F}, Y}$, where $\dot{F}$ represents the generic regressive
function added by $P_{X}$.


Let $\bar{D} = \langle d_{\delta} : \delta \in
C^{\lambda}_{\{\theta\}}\rangle$ be a sequence in $V$ such that each
$d_{\delta}$ is a cofinal subset of $\delta$ of ordertype
$\cf(\delta)$. For any set or class $Q$, we let ${^{\less\theta}Q}$ denote the set or class of functions
whose domain is an ordinal less than $\theta$ and whose range is
contained $Q$. We let $Ord$ denote the class of ordinals. 

A $V$-generic filter for $P_{X}$ is naturally represented by a pair
$(F,\bar{C})$, where $F$ is a regressive function on
$(C^{\lambda}_{[\theta,\lambda)})^{V}$, $\bar{C}$ has the form
$\langle C_{\alpha,\beta} : (\alpha,\beta) \in X\rangle$, and each
$C_{\alpha,\beta}$ is a club subset of $\lambda$ disjoint from
$F^{-1}\{\alpha\}$.  Fixing such a pair, we will define (in $V[F, \bar{C}]$) two models which
satisfy the theorem as $M$.

Let $M_{0}$ be $L(\bar{D}, F, {^{\less\theta}\{ C_{\alpha, \beta} : (\alpha, \beta) \in X\}}, {^{\less\theta}Ord})$. 
Let $M_{1}$ be the class of sets in 
$V[F, \bar{C}]$ which are hereditarily definable from the parameters 
$\bar{D}$, $F$, some member of ${^{\less\theta}Ord}$ and some member of ${^{\less\theta}\{ C_{\alpha, \beta} : (\alpha, \beta) \in X\}}$.
These are both 
models of ZF (see pages 182, 193 and 195-196 of \cite{Jec03}; note that 
$$M_{0} = \bigcup_{\gamma \in Ord}L(\bar{D}, F, {^{\less\theta}\{ C_{\alpha, \beta} : (\alpha, \beta) \in X\}}, {^{\less\theta}\gamma}),$$ and $M_{1}$ is an 
analogous union). 

Every set in $M_{0}$ is
definable in $M_{0}$ from $\bar{D}$, $F$, a member of
${^{\less\theta}Ord}$, the unordered set $\{ C_{\alpha, \beta} : (\alpha, \beta) \in X\}$ 
and a member of ${^{\less\theta}\{ C_{\alpha, \beta} : (\alpha, \beta) \in X\}}$.
It follows that $M_{0}$ is
closed under sequences of length less than $\theta$ in $V[F,
\bar{C}]$, and therefore that $M_{0}$ satisfies DC$_{\less\theta}$.
Since $\bar{D}$ is in $M_{0}$, and since $V$ and $V[F, \bar{C}]$ have
the same ordinals of cofinality less than $\theta$, $M_{0}$ satisfies
$\phi(\lambda, f)$, where $f$ is the ordinal successor function on
the regular cardinals below $\theta$.
Since $V[F, \bar{C}]$ and $V$ have the same sequences of ordinals of
length less than $\theta$, $M_{0}$ is definable in $V[F, \bar{C}]$ from
$\bar{D}$, $F$ and the (unordered) set $\{ C_{\alpha, \beta} :
(\alpha, \beta) \in X\}$.

Every set in $M_{1}$ is
definable in $V[F, \bar{C}]$ from $\bar{D}$, $F$, a member of
${^{\less\theta}Ord}$ and a member of ${^{\less\theta}\{ C_{\alpha, \beta} : (\alpha, \beta) \in X\}}$.
It follows that $M_{1}$ is
closed under sequences of length less than $\theta$ in $V[F,
\bar{C}]$, and therefore that $M_{1}$ satisfies DC$_{\less\theta}$.
Since $\bar{D}$ is in $M_{1}$, and since $V$ and $V[F, \bar{C}]$ have
the same ordinals of cofinality less than $\theta$, $M_{1}$ satisfies
$\phi(\lambda, f)$, where $f$ is the ordinal successor function on
the regular cardinals below $\theta$.
Since $V[F, \bar{C}]$ and $V$ have the same sequences of ordinals of
length less than $\theta$, $M_{1}$ is definable in $V[F, \bar{C}]$ from
$\bar{D}$, $F$ and the (unordered) set $\{ C_{\alpha, \beta} :
(\alpha, \beta) \in X\}$.


Given $Y \subseteq X$, let $N_{Y}$ denote $V[F, \langle
C_{\alpha,\beta} :(\alpha,\beta) \in Y \rangle]$.

\begin{lem}\label{twosix} Suppose that $X = Z \times Z$, for some 
$Z \subseteq \lambda$, and that $(F, \bar{C})$ is $V$-generic for
$P_{X}$. Then every subset of $V$ in $M_{0} \cup M_{1}$ exists in $N_{Y}$ for some
$Y \subseteq X$ of cardinality less than $\theta$.
\end{lem}

\begin{proof} Given such a set $A$, we can fix $Y \subseteq X$ of cardinality
less than $\theta$ such that $Y$ is of the form $W \times W$ for
some $W \subseteq \lambda$ and such that $A$ is definable in $V[F, \bar{C}]$ from
$F$, a set $x$ in $V$, $\{ C_{\alpha,\beta} : (\alpha,\beta) \in X\}$ and a function 
$h$ in $N_{Y}\cap{^{\less\theta}\{ C_{\alpha, \beta} : (\alpha, \beta) \in X\}}$.
Let $\phi$ be a
formula such that
$$A = \{ a \mid V[F, \bar{C}] \models \phi(a,F, x, \{ C_{\alpha, \beta} :
(\alpha, \beta) \in X\}, h)\}.$$ We have that $P_{X}$ is
forcing-equivalent to $P_{Y}
* Q_{\dot{F}, X \setminus Y}$. Suppose that there are two conditions
$d$ and $e$ in $Q_{F,X\setminus Y}$ (in $N_{Y}$) and some $a \in V$
such that $$d \forces \phi(\check{a}, \check{F}, \check{x}, \{ C_{\alpha, \beta} :
(\alpha, \beta) \in X\}, \check{h})$$ and $$e \forces
\neg\phi(\check{a}, \check{F}, \check{x}, \{ C_{\alpha, \beta} :
(\alpha, \beta) \in X\}, \check{h}).$$ There are conditions $d' \leq d$
and $e' \leq e$ in $Q_{F, X \setminus Y}$ such that \begin{itemize}
\item for every $(\alpha, \beta) \in \dom(d')$ there is a $\beta'$
such that $(\alpha, \beta')\in \dom(e')$ and $e'(\alpha, \beta') =
d'(\alpha, \beta)$, and
\item for every
$(\alpha, \beta) \in \dom(e')$ there is a $\beta'$ such that
$(\alpha, \beta')\in \dom(d')$ and $d'(\alpha, \beta') = e'(\alpha,
\beta)$.
\end{itemize}
There is then a natural isomorphism $\pi$ between $Q_{F,X \setminus
Y}$ below $d'$ and $Q_{F, X \setminus Y}$ below $e'$. This
isomorphism $\pi$ has the property that, given two generic filters
$G_{d'}$ and $G_{e'}$ for $Q_{F, X \setminus Y}$ with $\pi[G_{d'}] =
G_{e'}$, the (unordered) generic set $\{ C_{\alpha, \beta} :
(\alpha, \beta) \in X\setminus Y\}$ is the same in the two
extensions. Then $N_{Y}[G_{d'}] = N_{Y}[G_{e'}]$, and the set $\{ C_{\alpha, \beta} :
(\alpha, \beta) \in X\}$ is the same in these two extensions, 
contradicting the claim that $$d \forces \phi(\check{a}, \check{F}, \check{x}, \{ C_{\alpha, \beta} :
(\alpha, \beta) \in X\}, \check{h})$$ and $$e \forces
\neg\phi(\check{a}, \check{F}, \check{x}, \{ C_{\alpha, \beta} :
(\alpha, \beta) \in X\}, \check{h}).$$
\end{proof}


It follows from Lemma \ref{twosix} that every 
sequence of ordinals in $M_{0} \cup M_{1}$ of length less than $\lambda$
is in $V$, so $\lambda$ is
a regular cardinal in $M_{0}$ and in $M_{1}$. In the case that $X = \lambda \times
\lambda$, then, $M_{0}$ and $M_{1}$ each satisfy ZF + DC$_{\less\theta}$ +
$\phi(\lambda,f)$, where $f$ is the ordinal successor function on
the regular cardinals below $\theta$, and the function $F$ is in both models 
a regressive function on $C^{\lambda}_{[\theta, \lambda)}$ which is
not constant on a stationary set.



\end{document}